\documentclass[authoryear]{elsarticle}

\usepackage{natbib}
\usepackage{url}

\usepackage{latexsym}
\usepackage{amsmath}
\usepackage{amsfonts}
\usepackage{amssymb}
\usepackage{graphicx}

\usepackage{enumerate}

\input{def.tex}
\renewcommand{\eqref}[1]{(\ref{#1})}

\begin{document}

\begin{frontmatter}

  \title{Conditional fiducial models
  \footnote{
Accepted manuscript published as:
 Taraldsen G., Lindqvist B.H., Conditional fiducial models. J. Statist. Plann. Inference (2017),
\url{https://doi.org/10.1016/j.jspi.2017.09.007}.}}
\author[NTNU]{G.~Taraldsen}
\ead{Gunnar.Taraldsen@ntnu.no}
\author[NTNU]{B.~H.~Lindqvist}

\address[NTNU]{
Department of Mathematical Sciences,
Norwegian University of Science and Technology,
7046 Trondheim, NORWAY}

\begin{abstract}%
The fiducial is not unique in general,
but we prove that in a restricted class of models it
is uniquely determined by the sampling distribution of the data.
It depends in particular not on the choice of 
a data generating model.
The arguments lead to a generalization of 
the classical formula found by \citet{Fisher30}.
The restricted class
includes cases with discrete distributions,
the case of the shape parameter in the Gamma distribution,
and also the case of the correlation coefficient in a 
bivariate Gaussian model. 
One of the examples can also be used
in a pedagogical context to demonstrate 
possible difficulties with likelihood-, Bayesian-, and bootstrap-inference.
Examples that demonstrate non-uniqueness are also presented.
It is explained that they can be seen as cases with
restrictions on the parameter space.
Motivated by this the concept of a conditional 
fiducial model is introduced.
This class of models includes the common case
of {\it iid} samples from a one-parameter model investigated by \citet{Hannig13fiducial},
the structural group models investigated by \citet{FRASER},
and also certain models discussed by \citet{FISHER} in his final writing on the subject.

\end{abstract}


\vekk{
}

\end{frontmatter}


\newpage

\section{Introduction}

\citet[p.528-531]{Fisher30} introduced the term fiducial distribution of a parameter
by first criticizing the convention of using a non-informative Bayesian prior.
He explains that a constant prior is just as arbitrary as any other
prior by consideration of a constant prior for a reparametrized model.

There has been considerable progress since 1930 on the theory related to the choice
of a non-informative prior 
based on symmetry considerations \citep{Jef46},
entropy \citep{JAYNES}, 
or other information theoretic arguments \citep{Bernardo79refprior,BergerBernardoSun09refprior}.
The argument by \citet{Fisher30} remains, however, perfectly valid also today:
There is information in any particular choice for the prior.
 
As an alternative, for cases where there is lack of prior information,
Fisher introduced his fiducial argument.
According to \citet[p.428]{Fisher50cms}: 
{\em The importance of the paper lies, however, in setting forth a new mode of
reasoning from observations to their hypothetical causes.}
Today, this new mode of reasoning is still in development and
different lines of arguments have been published
\citep{SchwederHjort02Confidence,
TaraldsenLindqvist13fidopt,
XieSingh13confidence,
MartinLiu13fiducialbirnbaum,
HannigIyerLaiLee16review}.

The original argument \citep[p.532]{Fisher30} starts by consideration of the relation
\be{F1}
U = F(X \st \theta)
\ee
where $X$ is the maximum likelihood estimator of the parameter $\theta$, 
and $F (x \st \theta) = \pr^\theta (X \le x)$. 
From this \citet[p.534]{Fisher30} shows, given certain additional assumptions,
that the fiducial probability density $c(\theta)$ of $\theta$ 
is given by differentiation of $F$ with respect to $\theta$: 
\be{F2}
c(\theta \st x) = -\partial_\theta F(x \st \theta)
\ee
%
The assumptions ensure in particular that $U \sim \Uniform (0,1)$ so
discrete distributions are ruled out.

It is common to use the term {\em posterior} instead of the more cumbersome 
{\em Bayesian posterior distribution} when the meaning is clear from the context.
We will likewise use the term {\em fiducial} instead of the more cumbersome 
{\em fiducial distribution} when the meaning is clear from the context.

In this paper, we generalize formula~(\ref{eqF2})
to include cases with discrete distributions, 
and to cases where the initial equation~(\ref{eqF1}) is replaced
by an alternative equation which allows more general $U$.
This leads to a class of models where the fiducial is unique
and determined by the sampling distribution of the data as stated in Theorem~\ref{theo1}.
The generalization in this class of models is consistent with more general 
definitions of a fiducial model as considered
by 
\citet{TaraldsenLindqvist13fidopt}
and \citet{HannigIyerLaiLee16review}.

Motivated by cases where the conclusion in Theorem~\ref{theo1}
fails we introduce in Section~\ref{sCondFidUniqueness} 
the general concept of a conditional fiducial model
which is a generalization of the concept of a fiducial model
\citep{TaraldsenLindqvist13fidopt,HannigIyerLaiLee16review}:
It is a fiducial model with an added condition.
This is also motivated directly
by examples presented originally by \citet[p.138]{FISHER}
and also later by \citet{Seidenfeld92fiducial}.
Conceptually, this can be seen as the most important result in this paper. 

It is shown, by general examples, 
that this gives a possible
approach for fiducial inference in multivariate observations and for group models.
In the group case a natural condition is given by a maximal invariant,
and the inference from the conditional fiducial model 
then coincides with the Bayes posterior from the right Haar prior.
For the conditional fiducial models considered the resulting fiducial is again unique,
but only when both the fiducial model and the condition is given.
The fiducial is, however, not uniquely given by the sampling distribution of the data.


\vekk{

One major source of non-uniqueness of the fiducial is that
the fiducial argument based on equation~(\ref{eqF1}) can be generalized
in many different directions, 
and hence the term {\em fiducial} is used with different interpretations by
different authors.

Historically, up to the present date, 
the most important direction is given by the theory of confidence intervals 
and hypothesis testing as developed by \citet{Neyman37confinterval}.
This line of arguments defines the fiducial based on pivotal quantities,
and recent developments are discussed briefly at the end of this paper in Section~\ref{sDiscussion}.
Section~\ref{sUnique} gives
a definition of the fiducial that is not based on a pivotal,
and formulates Theorem~\ref{theo1} which is the main technical result in this paper.
Examples are presented in Section~\ref{sExamples} and
proofs are presented in Section~\ref{sProof}.
Section~\ref{sCondFidUniqueness} introduces the concept of
a conditional fiducial model and discuss it with examples
linking this to conditional inference based on ancillary statistics.
}

\section{Fiducial inference}
\label{sFiducial}

Fiducial theory, as presented here, is closely connected to
methods for simulation and testing of statistical inference procedures.
This is by itself good motivation for the current interest in fiducial theory. 
Computer simulation of models has become an integral part of
modern statistical practice, and fiducial arguments are
particularly well adapted to this practice.

{\it A fiducial model} $(U, \chi)$, as defined in this paper, 
is defined by a fiducial equation
\begin{equation}
\label{eqFidEq}
x = \chi (u, \theta)
\end{equation}
for an {\it observation} $x$ 
together with a probability distribution $\pr_U^\theta$ 
for $u$ given the {\it parameter} $\theta$.
The definition given by equation~(\ref{eqFidEq}), 
and its relation to other possible definitions,
is discussed in more detail by \citet{TaraldsenLindqvist16condprob,TaraldsenLindqvist15fidpost,TaraldsenLindqvist13fidopt}.
The relation~(\ref{eqFidEq}) replaces the corresponding relation
$u=F(x,\theta)$ in equation~(\ref{eqF1}) used originally by
\citet{Fisher30}. 
The distribution $\pr_U^\theta$ is
the Monte Carlo law of the fiducial model.
Sampling from the statistical model $\{\pr_X^\theta\}$ is defined by
$x=\chi (u,\theta)$ for a simulated sample $u$ from $\pr_U^\theta$ 
and a fixed model parameter $\theta$.
A fiducial model is hence exactly what is needed for
simulating data from a statistical model:
It is a data generating model.

Sampling from the fiducial is 
defined by solving the fiducial equation 
to obtain $\theta = \htheta (u,x)$
for each simulated sample $u$ and the given observed $x$.
Existence and uniqueness of the solution is assumed,
and the fiducial model is then said to be {simple}.
For a {\em simple fiducial model} it is also 
assumed that the Monte Carlo law $\pr_U^\theta$
does not depend on $\theta$.
More general cases can be considered,
but the theory then splits into several alternative theories 
\citep{Dempster68bayesgen,
Shafer08belief,
FRASER,Wilkinson77fiducial,
DawidStone82,TaraldsenLindqvist13fidopt,
MartinLiu13fiducialbirnbaum,HannigIyerLaiLee16review}.
\vekk{\citet{TaraldsenLindqvist16condprob} explains that the theory as
presented here can be formulated naturally within 
the theory of conditional probability spaces as formulated by \citet{RENYI}.
}


\vekk{
The end result is then a fiducial distribution for the
unknown model parameter.
The fiducial distribution summarizes the knowledge
obtained from the fiducial model and the observations.
This distribution has hence an interpretation
similar to a Bayesian posterior,
but it is obtained without a prior distribution for the parameter.
}

Formally, in the case of a simple fiducial model,
the fiducial distribution can be defined as follows.
Let $U^x \sim \pr_U^\theta$ and define the fiducial
random quantity $\Theta^x = \htheta (U^x, x)$.
The law of $\Theta^x$ is the fiducial distribution or simply the fiducial.
The interpretation is as for the posterior in a Bayesian
analysis: It is an epistemic probability law derived for the parameter based
on the observation.
The role of the prior and the statistical model in Bayesian analysis
is replaced by use of the fiducial model~(\ref{eqFidEq}) in 
fiducial inference. 
The fiducial is obtained in this case without a prior distribution for the parameter.

A further advantage of the fiducial approach in the case of a simple 
fiducial model is that independent 
samples are produced directly from independent sampling from $\pr_U^\theta$.
Bayesian simulations most often come as dependent samples from a Markov chain.
\citet{FRASER} presents many non-trivial models which can be used
to exemplify this for applied concrete problems.

The fiducial argument can also be used within a frequentist frame of inference.
This is similar to how the Bayesian machine is used to produce
frequentist methods in cases which would otherwise be intractable.
The fiducial approach has then the advantage that
it comes equipped with a method for 
simulating data from the statistical model. 
This is exactly what is needed for repeated testing of any given
inference procedure suggested for the model: 
Fiducial, Bayesian, or obtained by other arguments.
This is a most convenient circumstance.
\vekk{, even though
the testing can be quite computer intensive if
the inference procedure itself is based on the fiducial distribution:
For each simulation of data from the model a corresponding
repeated sampling from the fiducial is required. 
}

It is clear, and this will be exemplified in the next section,
that there exist many different fiducial models for a 
given statistical model.
This is related to the fact that there exist many different algorithms 
for the simulation of data from a given statistical model.

The concept of a fiducial model as used in this paper
is hence not so that it is uniquely determined by the likelihood function.
\citet{FISHER}
insisted that the fiducial distribution should be defined in terms of
the likelihood function.
An example is given in section~{\ref{sExamples}} of a statistical model
where the likelihood is not defined,
and a definition based on the likelihood is then impossible.
This explains why we, and many other authors \citep{Dempster68bayesgen,FRASER,
DawidStone82,TaraldsenLindqvist13fidopt,
MartinLiu13fiducialbirnbaum,HannigIyerLaiLee16review},
take equation~(\ref{eqFidEq}) as the initial starting point for fiducial inference.

\section{A unique fiducial}
\label{sUnique}

Assume now that two different fiducial models exist for a given statistical model
$\{\pr_X^\theta\}$. 
A natural question is then:
{\em Do two fiducial models give the same fiducial distribution?}
An affirmative answer to this question follows when 
restricted to the case where both $x$ and $\theta$
are real numbers,
and where $\theta \mapsto \chi(u,\theta)$
is either strictly increasing for all $u$ or strictly decreasing for all $u$.
The fiducial model is then said to be {\em strictly monotonic}.
It is here assumed that the parameter set $\Omega_\Theta$ and the set of possible
observations $\Omega_X$ are subsets of the real line,
or that they are order isomorphic to subsets of the real line.
This includes in particular discrete subsets of the  real line.
The Monte Carlo space $\Omega_U$ is a general measurable space.
\begin{theo}
\label{theo1}
The fiducial distribution of a
strictly monotonic 
simple fiducial model is
uniquely determined by
the sampling distribution of the data.
If, additionally, the sampling distribution is continuous,
then the fiducial distribution is an exact confidence distribution.
\end{theo}
%


The last statement in Theorem~\ref{theo1} means that
$\pr^\theta (\theta \le \theta_{1 - \alpha} (X)) = 1 - \alpha$
where $\theta_{1-\alpha} (x)$ is the $(1-\alpha)100$ percentile
of the fiducial distribution \citep{SchwederHjort16book}.
Theorem~\ref{theo1} will be proved here as a consequence 
of two propositions established next.
The proofs are elementary, 
but all details are given since
the results are novel and of interest by themselves.
It should be observed that there are no restrictions on
the Monte Carlo distribution except that it should
be a probability distribution.
It need in particular not be continuous,
and can in fact be infinite dimensional as
sometimes formally required in Markov Chain Monte Carlo simulations.
The first proposition gives
Fisher's original result
$c(\theta \st x) = -\partial_\theta F(x \st \theta)$ 
for the fiducial density as a special case.
\begin{prop}
\label{prop1}
Assume that 
$(U, \chi)$ is a simple fiducial model such that
$\theta \mapsto \chi (u,\theta)$ is strictly increasing for all $u$,
and let $F(x \st \theta) = \pr^\theta (\chi (U, \theta) \le x)$.
The cumulative distribution function of the fiducial 
given an observation $x$ is then
\be{CIncreasingRepeat}
C (\theta \st x) = 1 - \lim_{\epsilon \downarrow 0} F(x-\epsilon \st \theta) = 1 - F(\xm \st \theta)
\ee
\end{prop}
\begin{proof}
The fiducial distribution is determined by
the following calculation
\begin{subequations}
\label{eqFid3}
\begin{alignat}{2}
C (\theta \st x)
& = \pr^\theta (\htheta (U, x) \le \theta) & & \qquad \text{Fiducial CDF}\\
& = \pr^\theta (x \le \chi (U, \theta)) & & \qquad \text{Strictly increasing}\\
& = \pr^\theta (x \le X) & & \qquad \text{Fiducial model}\\
& = 1 - \pr^\theta (X < x) & & \qquad \text{Complement}\\
& = 1 - \pr^\theta (X \le x) + \pr^\theta (X = x) & & \qquad \text{Could be discrete}\\
& = 1 - F(x \st \theta) + (F(x \st \theta)-F(\xm \st \theta)) & & \qquad \text{Limit for atom}\\
& = 1 - F(\xm \st \theta) & & \qquad \text{Cancellation}
\end{alignat}
\end{subequations}
\hfill $\Box$
\end{proof}
If it is assumed additionally 
that $X$ has a continuous distribution
and that the cumulative distribution function $F(x \st \theta)$
is differentiable with respect to the parameter,
then the fiducial density (\ref{eqF2}) of Fisher  follows as a special case of
(\ref{eqCIncreasingRepeat}).
Formula~(\ref{eqCIncreasingRepeat}) seems to be a novelty in the case 
where atoms are allowed in the statistical model, 
and in particular for a discrete sample space. 
It follows also as a consequence that 
$\theta \mapsto F(\xm \st \theta)$ is continuous from the right and
decreasing from $1$ down to $0$.
These necessary conditions on $F$ 
seem also to be a new observation.

%
\begin{prop}
\label{prop2}
Assume that 
$(U, \chi)$ is a simple fiducial model such that
$\theta \mapsto \chi (u,\theta)$ is strictly decreasing for all $u$,
and let $F(x \st \theta) = \pr^\theta (\chi (U, \theta) \le x)$.
The cumulative distribution function of the fiducial 
given an observation $x$ is then
\be{CDecreasing}
C (\theta \st x) = F(x \st \theta)
\ee
\end{prop}
\begin{proof}
The fiducial distribution is determined by
the following calculation
\begin{subequations}
\label{eqFid3}
\begin{alignat}{2}
C (\theta \st x)
& = \pr^\theta (\htheta (U, x) \le \theta) & & \qquad \text{Fiducial CDF}\\
& = \pr^\theta (\chi (U, \theta) \le x) & & \qquad \text{Strictly decreasing}\\
& = \pr^\theta (X \le x) & & \qquad \text{Fiducial model}
\end{alignat}
\end{subequations}
\hfill $\Box$
\end{proof}
In this case it follows as a consequence that
$\theta \mapsto F(x \st \theta)$ is continuous from the right and
increasing from $0$ up to $1$.
\vekk{
In both the decreasing and increasing cases 
it cannot be concluded that the conditions on $F$ are
sufficient in the sense that it will correspond to
a conventional strictly monotonic fiducial model:
An initial problem is that monotonicity of $\theta \mapsto F(x \st \theta)$
cannot be used to prove existence of a strictly monotonic fiducial model.
}

\vekk{The conditions are, however, sufficient for the definition of a CDF
$C$ directly,
and this gives then a natural generalization of a fiducial
distribution for these cases.
}

\begin{proof}{\bf (of Theorem~\ref{theo1})}
Propositions~\ref{prop1}-\ref{prop2} prove the claimed uniqueness of the fiducial
as stated. 
For the strictly increasing case the exact confidence statement follows 
as in Fisher's original argument:
\be{ConfStatement}
\begin{split}
\alpha & = \pr^\theta \left(F (X \st \theta) < \alpha \right) \\ 
& = \pr^\theta \left(X < F^{-1}(\alpha \st \theta) \right) \\
& = \pr^\theta \left(\hat{\theta}_F (\alpha,X) < \theta \right) 
\end{split}
\ee
where it is assumed that 
the law of $X$ contains no atoms.
The mapping
$\theta \mapsto \chi_F (u_F,\theta) = F^{-1} (u_F, \theta)$ is strictly increasing
since $\theta \mapsto F(x \st \theta) = \pr (\chi(U,\theta) \le x)$
is strictly decreasing from the assumption of
a strictly increasing $\theta \mapsto \chi (u,\theta)$.
Equation~(\ref{eqConfStatement}) implies that
the interval statistic $(-\infty,\hat{\theta}_F (\alpha,x))$ is 
an exact $(1-\alpha)$-confidence interval.
On the other hand
\be{ConfStatement1}
\begin{split}
\pr^\theta (\hat{\theta}_F (U_F,x) \le \hat{\theta}_F (\alpha,x)) =
\pr^\theta (U_F > \alpha) = 1-\alpha
\end{split}
\ee
so the upper limit $\hat{\theta}_F (\alpha,x)$
is identical with the $(1-\alpha)100$ percentile
of the fiducial distribution of 
$\hat{\theta}_F (U_F,x)$.
Since the fiducial distribution is unique it follows
that the upper limit $\hat{\theta}_F (\alpha,x)$
coincides with the $(1-\alpha)100$ percentile
$\theta_{(1-\alpha)} (x)$ of the fiducial distribution of
$\hat{\theta}(U,x)$.
This proves that the fiducial distribution is a confidence distribution in this case.
The strictly decreasing case is similar.
\hfill $\Box$
\end{proof}

The correspondence with the fiducial model based on the
cumulative distribution function 
gives as a by-product that the fiducial distribution in this case is a confidence distribution.
As explained above this argument holds more generally for any strictly monotonic simple fiducial model,
and proves that the fiducial is then 
a confidence distribution in the case of continuous variables.

An elegant direct proof of this follows from the work of
\citet{BolvikenSkovlund96conf} and \citet{LillegaardEngen99confidence}.
The proof presented here 
is more elementary since it is sufficient to consider the original result of Fisher
based on the pivotal $F (X \st \theta) \sim \Uniform (0,1)$.

Theorem~\ref{theo1} is the main technical result of this paper.
It should be observed that the first part is also valid for 
discrete, and more general, distributions.
The resulting fiducial can still be used to define confidence
distributions, but these are then by necessity not exact.

It should be observed that Theorem~\ref{theo1} assumes initially
that there is a strictly monotonic and simple fiducial model,
and then the fiducial is determined by the sampling distribution of the statistic.
This opens up for the possibility of alternative fiducial distributions
based on a fiducial model which is not strictly monotonic.
Additionally, the fiducial will depend on the choice of a statistic.
This is exemplified in the following sections where four possible fiducial distributions
are found for the shape parameter in a gamma distribution.

The usefulness of Theorem~\ref{theo1} as compared with the original paper
by \citet{Fisher30} is that the Monte Carlo variable $U$ is not restricted
to be real and uniformly distributed on $(0,1)$.  
This means that the sampling distribution of the data need not be restricted to be continuous.
Furthermore, the freedom of choice of a suitable Monte Carlo variable can give simplifications. 
A demonstration of the latter will be given in the next section with a three dimensional and an 
$n$ dimensional Monte Carlo variable for the case of inference for 
respectively the correlation coefficient problem of Fisher
and the case of a gamma shape parameter.

\section{Examples with uniqueness and non-uniqueness}
\label{sExamples}

Four examples are presented in this section.
The first two examples concern the correlation coefficient of
a Gaussian distribution and the shape parameter of a gamma distribution.
The resulting exact confidence distributions are
known in the literature,
but are seldom used in applications.
The uniqueness of these fiducial distributions as a result of
Theorem~\ref{theo1} together with the
simple algorithms presented should encourage more widespread use
in applications.
This includes hypothesis testing based on the
resulting confidence intervals and
when reporting uncertainty,
but also for the purpose of providing alternative
estimators adapted to specific loss functions\citep{TaraldsenLindqvist13fidopt}.

The third example gives a discrete model for digitized data.
This exemplifies that cases without a likelihood function
can occur,
and also exemplifies that Theorem~\ref{theo1} gives
a fiducial distribution also for discrete distributions.
\citet{FISHER} restricted his definition of the 
fiducial to the continuous case,
but 
acknowledged simultaneously that essentially all 
existing data sets come as digitized observations 
\citep{Taraldsen06GUMResolution,CisewskiHannig12linearmodels,HannigIyerWang07fiducialquant}.

The fourth example follows from \citet{FISHER}~'s 
final discussion of fiducial inference,
and exemplifies that for non-simple models the
conclusion of Theorem~\ref{theo1} may fail:
There are two reasonable candidates for the fiducial distribution.
It also demonstrates, as do the first two examples,
that the fiducial distribution can be different from a Bayesian posterior.
General conditions that ensure that the fiducial equals a Bayesian posterior
have been obtained recently \citep{TaraldsenLindqvist15fidpost}.

\subsection{The correlation coefficient}
\label{ssCorr}

Consider the strictly increasing simple fiducial model
\be{FidCorr}
x = \frac{\theta u_1 + u_3}{u_2},\;\;
u_1 \sim \chi^2_{n-1}, 
u_2 \sim \chi^2_{n-2}, 
u_3 \sim \Normal (0,1),
\ee
where the components of $u$ are independent.
The notation $u_1 \sim \chi^2_{n-1}$ means that
$u_1$ is a realization from the $\chi^2_{n-1}$ distribution,
and likewise for the other components of $u$.

The fiducial equation~(\ref{eqFidCorr}) can be solved for $\theta$ to give
\be{FidCorrDist}
\theta = \frac{x u_2 - u_3}{u_1},
\ee
for a given observed $x$.
This determines the fiducial as explained in 
Section~{\ref{sUnique}}.

The results of \citet{LindqvistTaraldsen05}
can be used to prove that
the one-one correspondence $x \leftrightarrow r$
\be{FidCorr1} 
r=\frac{x}{\sqrt{1 + x^2}},\;\;
x=\frac{r}{\sqrt{1-r^2}}
\ee
defines a variable $r$
distributed like the empirical 
correlation coefficient of a bivariate 
Gaussian sample of
size $n$ with a correlation coefficient 
\be{FidCorr2} 
\rho = \frac{\theta}{\sqrt{1 + \theta^2}}
\ee
The fiducial distribution of $\rho$ is determined
by the fiducial distribution of $\theta$
and equation~(\ref{eqFidCorr2}).
Altogether, 
this gives a simple method for simulation of
independent samples from the fiducial distribution of the 
correlation coefficient.

An alternative fiducial model
\be{Corr2}
r = F^{-1} (u \st \rho),\;\;\; u \sim \Uniform (0,1)
\ee
follows by inversion of the cumulative distribution
function $F (r \st \rho)$ of the empirical correlation function.
This fiducial model was the one \citet{Fisher30} used
when introducing the fiducial argument.
He tabulated critical values for a case with sample size $n=4$
based on his distribution formula \citep{Fisher15CorrCoeff}.

A numerical simulation from the resulting fiducial distribution
will require more complicated numerical methods than the
explicit solution given by equation~(\ref{eqFidCorr2}):
The cumulative distribution function of the empirical correlation function
and its inverse is missing in standard numerical libraries.
Additionally, equation~(\ref{eqCorr2}) must be solved numerically for
$\rho$ for each simulated $u$.
A somewhat simpler and equivalent approach is to solve
$u = F(r \st \rho)$ for $\rho$ for each simulated $u$.

It is not initially clear, at least according to the present authors,
that the fiducial distribution determined from
equation~(\ref{eqCorr2}) coincides with the fiducial distribution
determined by equation~(\ref{eqFidCorr2}).
Theorem~\ref{theo1} ensures, however, that the resulting fiducial distributions
for the correlation coefficient are identical,
and that the fiducial is a confidence distribution. 
The latter claim follows in this case directly from the original arguments of
\citet{Fisher30} applied on equation~(\ref{eqCorr2}).

The method given above by equation~(\ref{eqFidCorrDist}) 
was presented by \citet[p.1056]{DawidStone82}.
This problem was also considered by
\citet{BolvikenSkovlund96conf} and
\citet{LillegaardEngen99confidence},
but without explicit identification with the fiducial obtained
originally by \citet[p.534]{Fisher30} for the correlation coefficient.
We consider this identification, 
and its generalization in the form of Theorem~\ref{theo1},
to be principally and practically very important.

\citet{Fisher30} calculated as stated above some selected 
percentiles $\rho_{1-\alpha} (x)$ of the fiducial for the correlation coefficient.
It follows from the arguments just given, 
and this is well known in the literature,
that these percentiles are found by solving 
$\alpha = F (x \st \rho_{1-\alpha})$ for $\rho_{1-\alpha}$. 
This is much simpler than the approach given by
simulating from the fiducial,
and explains how this was possible without use of computers.
The reward of simulating from the fiducial is that 
it gives information regarding all levels of confidence simultaneously.
This is similar to the argument in favor of using
p-values in hypothesis testing.
\citet{SchwederHjort02Confidence} and
\citet{XieSingh13confidence} explain in detail that
this is more than a similarity,
and that it is essentially an equivalence.

The obtained uniqueness of the fiducial for $\rho$ is
granted by assuming that the inference is based on the
empirical correlation coefficient.
The fiducial is not unique when the inference is to be based on
the multivariate sample itself since this is not a simple fiducial model
as discussed more generally in section~\ref{sCondFidUniqueness}.
An interesting case is presented by
\citet{Reid03asymptotics} who considered asymptotic inference for the
correlation coefficient for the case where 
the mean and variance is known.
The inference is then based on a two-dimensional minimal sufficient statistic. 
It would be interesting to compare this procedure with the suggested procedure.

\subsection{The Gamma shape parameter}
\label{ssGamma}

A fiducial model for a random sample of size $n$ from the 
gamma distribution with scale parameter $\beta$ and shape parameter $\alpha$
is given by 
\be{FidBartlett1}
y_i = \beta F^{-1}(u_i \st \alpha),\;\; u_i \sim \Uniform (0,1)
\ee
where the components of $u$ are independent and
$F^{-1} (u \st \alpha)$ is the inverse cumulative distribution
function of a gamma variable with shape $\alpha$ and scale $\beta=1$.

Let $\overline{y}$ and $\widetilde{y}$ be the arithmetic and geometric means.
The Bartlett statistic $w$ is defined as the fraction
$w = \widetilde{y}/\overline{y}$.
Equation~(\ref{eqFidBartlett1}) gives the following fiducial model \citep{TaraldsenLindqvist13fidopt}
\be{FidBartlett}
w = \widetilde{F^{-1}}(u_i; \alpha) / \overline{F^{-1}}(u_i; \alpha)
\ee
%
This is a strictly increasing simple fiducial model.
Sampling from the fiducial can be done by simulation from the uniform distribution, 
and solving equation~(\ref{eqFidBartlett}) numerically for each sample. 

An alternative fiducial model for the Barlett statistic
is given by inversion of its cumulative distribution function.
The Monte Carlo law will then simply be the uniform distribution on $(0,1)$,
but the cumulative distribution function of the Bartlett statistic
is not available in standard numerical libraries.
This implementation of the fiducial distribution for the
shape parameter is hence less straightforward than the one given by
equation~(\ref{eqFidBartlett}).
Theorem~\ref{theo1} ensures, however, that the two 
resulting fiducial distributions
for the shape parameter are identical,
and that it is a confidence distribution.

It can be observed that the arguments just presented for the
gamma shape parameter are very similar to the
arguments given for the correlation coefficient.
In both cases the inversion method gives a fiducial model directly
for the parameter of interest, 
and the exact confidence property 
holds due to the original arguments given by \citet{Fisher30}.
Simulation from this fiducial model is possible,
but an alternative fiducial model gives a simplified algorithm.

Regarding uniqueness the situation is again similar to
the correlation coefficient case:
The obtained uniqueness of the fiducial for $\alpha$ is
granted only by assuming that the inference is based on the
Bartlett statistic.
The fiducial is not unique when the inference is to be based on
the initial multivariate fiducial model~(\ref{eqFidBartlett1}).
This is discussed in some more generality in Subsection~\ref{ssRepeated}
which also gives three other alternative fiducial distributions for $\alpha$.

A similar approach can be used on many other models including in particular
the exponential family investigated by \citet{VeroneseMelilli14fiducial}.   
They follow Fisher and define the fiducial by 
the cumulative distribution function, 
and investigate into
explicit and asymptotic expressions for the fiducial density.
Theorem~\ref{theo1} applies for these cases also,
and it can hence be used to find alternative methods for
simulation from the fiducial as exemplified by the correlation coefficient and 
shape parameter.

\blind{
The Bartlett statistic depends only on the shape parameter $\theta=\alpha$
and this is here identified as the model parameter for simplicity.
\citet{TaraldsenLindqvist13fidopt} considered the case where
$(\alpha, \beta)$ is the model parameter,
but the resulting inference for the shape $\alpha$ is identical.
The combined results give a joint confidence distribution for
$(\alpha,\beta)$ where the marginal distribution for the shape
$\alpha$ is also a confidence distribution.
}


\subsection{Digitized data}
\label{sDiscrete}

\citet{FISHER} explicitly stated that the fiducial argument
was not to be used for discrete distributions.
One reason for this is that his argument fails in this case since
$F (X \st \theta)$ is not uniformly distributed when $X$
is discrete.
The following demonstrates that the fiducial defined
by equation~(\ref{eqFidEq}) can handle the discrete case.
Fisher's original result
$c(\theta \st x) = -\partial_\theta F(x \st \theta)$
for the fiducial density cannot be used,
but a more general formula based on the cumulative distribution
$F (x \st \theta)$ is demonstrated to be valid.

Incidentally, the example also demonstrates that likelihood inference
is not generally available:
There is no likelihood corresponding to this model since
the distribution is not given by a density with respect to a measure
that does not depend on the parameter.
This means that a definition of the fiducial based on 
the likelihood will fail in this case.

Let the fiducial model be the location model
\be{FidLoc}
x = u + \theta, 
\ee
where $x$, $\theta$ are arbitrary real numbers
and the Monte Carlo law of $u$ is discrete
and supported on $\ldots,-2d,-d,0,d, 2d, \ldots$ where
$d>0$ is the digital resolution.
The resulting statistical model is supported on
$\theta + \{\ldots,-2d,-d,0,d, 2d, \ldots\}$, 
and the likelihood cannot be defined in the usual way from a density.

The fiducial distribution is simply given by solving equation~(\ref{eqFidLoc})
with respect to $\theta$.
This gives the location model
\be{FidLocTheta}
\theta = x - u, 
\ee
and it follows that the 
cumulative distribution function 
$C (\theta \st x)$ of the fiducial
is a right continuous step function with steps at
$x + \{\ldots,-2d,-d,0,d, 2d, \ldots\}$.
It is also given, as the reader can verify by inspection, 
by the formula
\be{CIncreasing}
C (\theta \st x) = 1 - \lim_{\epsilon \downarrow 0} F(x-\epsilon \st \theta) = 1 - F(\xm \st \theta)
\ee
proved more generally in Section~\ref{sUnique}.
This formula coincides with Fisher's original definition
for the density $c(\theta \st x) = -\partial_\theta F(x \st \theta)$
in the case of a continuous distribution,
but this example demonstrates the need for taking a limit
from the left in the general case including atoms in the distribution.

\vekk{
Incidentally,
equation~(\ref{eqFidLoc})
also exemplifies a statistical model that is not given by a
discrete or continuous probability density,
and then also a model where the likelihood function is not definable
in the usual sense.
This means that a definition of the fiducial based on 
the likelihood will fail in this case.
}

If, however, the model assumptions are changed so that
$x$ and $\theta$ are restricted to
$\{\ldots,-2d,-d,0,d, 2d, \ldots\}$,
then the statistical model is given by a density,
and the likelihood coincides with the fiducial.  

\vekk{
A bell shaped distribution can then furthermore be used 
to obtain a model for
digitized data with a precision given by $d > 0$.
}
\vekk{A reviewer suggested that this model is more
natural than the previous.
We agree, but the previous model is included since it 
demonstrates existence of a model were the likelihood function is undefined
as explained above.
}

Assume that the Monte Carlo law of $u$
is skewed to the right.
The model in equation~(\ref{eqFidLoc}) then illustrates that the
parametric bootstrap distribution,
which here equals the distribution of $x$ which is taken as an estimator of $\theta$, 
is skewed in the opposite direction
of the confidence distribution given by the fiducial.
Finally, it illustrates that the optimal estimator given by the mean
of the fiducial \citep{TaraldsenLindqvist13fidopt} is different from
the maximum likelihood.
By suitable choices for the Monte Carlo law it can be
illustrated that the maximum likelihood can be arbitrary far off from
the optimal.

All together, 
the simple location model above can be used to illustrate
many theoretically important possibilities.  
The statistical analysis of digitized data is, however, 
a very important practical problem that deserves 
attention not only from the digital signal processing community,
but also from statisticians.

\subsection{A case with two candidate fiducial distributions}
\label{sTwoFiducials}

The previous three examples demonstrated usage of
Theorem~\ref{theo1}.
The example presented next 
exemplifies what can happen if one of the assumptions in 
Theorem~\ref{theo1} is not fulfilled.
It gives a case where two candidate fiducials appear,
and also illustrates how restrictions on the parameter space can be handled
by two different motivations:
%
\begin{enumerate}[a)]
\item The fiducial as an epistemic probability similar to a Bayesian posterior.
\item The fiducial as a confidence distribution. 
\end{enumerate}

Let $x_1, \ldots, x_n$ be a random sample from a Gaussian distribution
with mean $\mu$ and variance $\sigma^2$.
Assume that the variance is known, 
and base the inference on the minimal sufficient statistic $\overline{x}$.
A fiducial model is then given by
\be{fx}
\overline{x} = \mu + \sigma \overline{u}
\ee
where $\sigma \overline{u}$ is Gaussian with
mean $0$ and variance $\sigma^2 /n$. 
The fiducial is determined by
\be{fmu}
\mu = \overline{x} - \sigma \overline{u}
\ee
and is hence Gaussian with mean $\overline{x}$ and variance $\sigma^2 /n$.

Assume next as above, 
but with the added restriction $\mu < \mu_{\max}$ 
with a specified upper bound $\mu_{\max}$.
In this case it will be explained that there are two natural candidates 
for the fiducial.

One possible general strategy is to sample from the fiducial by
ignoring the cases where a simulated $u$ does not lead 
to a solution of the fiducial equation.
In the present example this leads to the
fiducial as the previous fiducial (\ref{eqfmu}),
but conditioned on $\mu < \mu_{\max}$.
It is the natural choice when the fiducial is interpreted as
a state of knowledge regarding the unknown parameter.
The additional knowledge given by restrictions on the parameter
is included by conditioning on the restriction.
This will be discussed in more generality in Section~\ref{sCondFidUniqueness}.

The resulting distribution is also the fiducial distribution that would follow
by arguments in recent publications \citep{HannigIyerLaiLee16review}.  
It equals the Bayesian posterior obtained from a normalization of the likelihood,
and inherits corresponding properties regarding optimal inference.  
This fiducial is hence equal to a Gaussian distribution normalized to an interval.

Another candidate is obtained by keeping the obtained density 
when $\mu < \mu_{\max}$ and placing a point mass at the terminal point $\mu=\mu_{\max}$
to ensure normalization.
The point mass is hence set equal to the probability of not obtaining a 
solution of the original fiducial model.
This fiducial coincides with the
fiducial as stated by \citet[p.140,eq.116]{FISHER}
in a more complicated example discussed in more detail below.
It can be shown that this fiducial is a confidence
distribution in a natural sense. 

The conclusion is
that for the case of a semi-infinite line there are two competing candidates
that could qualify for the title as being 'the fiducial'. 
Guidance by the aim of obtaining confidence distributions gives the fiducial suggested by Fisher,
but guidance by similarity with a Bayesian posterior gives an alternative. 
The cause of the difficulty lies in the observation that the fiducial model is 
no longer simple when restrictions are put on the parameters, 
and the corresponding failure of existence of a solution $\htheta (u,x)$
for all pairs $(u,x)$ can be handled in different ways. 

This example can be generalized to the case of different interval 
restrictions and to different underlying sampling distributions.
The choice of the Gaussian distribution gave the possibility of 
explicit reference to \citet[p.140,eq.116]{FISHER}.
The argument can also be generalized into a new class of fiducial models as explained next.

\section{Conditional fiducial models}
\label{sCondFidUniqueness}

\vekk{
This section considers more closely 
\citet{FISHER}~'s discussion of fiducial distributions for the case
of restrictions on the parameters.
It is in particular explained that non-uniqueness can occur 
as a consequence of a so called Borel paradox.
Finally, it is explained how uniqueness can be restored in these cases.
}

A {\em conditional fiducial model} $(U, \chi, C)$ 
is defined by a fiducial model $(U, \chi)$
as defined by equation~(\ref{eqFidEq}) together with a condition
\be{CondFid}
C (\theta) = c
\ee
It is here assumed that $C: \Omega_\Theta \into \Omega_C$ is a 
measurable function from the model parameter space $\Omega_\Theta$ into the
set $\Omega_C$ where $c \in \Omega_C$.
Assume that $\Theta^x$ is a random quantity with a distribution equal to
a fiducial distribution obtained by inference based on the fiducial model and the data $x$.
This distribution is unique in the case of a simple fiducial model,
but good candidates exists also in other special cases as described by references in the Introduction. 
Based on a specific $\Theta^x$ the conditional distribution $(\Theta^x \st C=c)$ defines then 
the unique fiducial distribution of the conditional fiducial model.

Consider again the case of {\it iid} sampling from
a normal distribution with known variance and unknown mean $\theta$. 
The knowledge $\theta < \theta_{\max}$ can be realized by 
the indicator function $C (\theta) = (\theta < \theta_{\max})$ and
letting $c=1$.
In this case it is only the single level set 
$\{\theta \st C(\theta) = c\}$ that matters since
$\Pr (C(\Theta^x)=c)>0$.
In general the function $C$ must be specified to 
obtain a well-defined conditional fiducial
to avoid a Borel paradox \citep[p.50]{KOLMOGOROV}.  

\subsection{Parameters restricted to a curve}
\label{ssNonUniquenessFisher}

Consider the case where 
the observation $x$ is given by
a fiducial model
\be{FishLocSrc}
x = \theta + u
\ee
where $x$, $\theta$ and $u$ belong to a Hilbert space.
The fiducial distribution is given by 
\be{FishLoc1}
\theta = x - u
\ee
and the Monte Carlo law for $u$.
A condition $C(\theta) = c$ determines then a unique
conditional fiducial.
The general case of a Hilbert space was considered by \citet{TaraldsenLindqvist13fidopt}
and will also be discussed later in this paper.

This example, 
in the two-dimensional case,
 was treated by Fisher and
motivated the definition of a conditional fiducial model as defined before equation~(\ref{eqCondFid}).
\citet[p.138]{FISHER} considers the case 
where the Monte Carlo law of $u$ in equation~(\ref{eqFishLocSrc}) is a standard bivariate normal and
$\theta$ is the unknown mean identified with a point in the plane.
Fisher considers next the following cases given by
the additional knowledge that $\theta$ lies on:
%
\begin{itemize}
\item A straight line.
\item A circle with radius $R$.
\item A curve in the plane.
\end{itemize}
%
This is illustrated with a more general joint fiducial in Figure~\ref{fig1}
\begin{figure}
\centering
\includegraphics[width=0.9\textwidth]{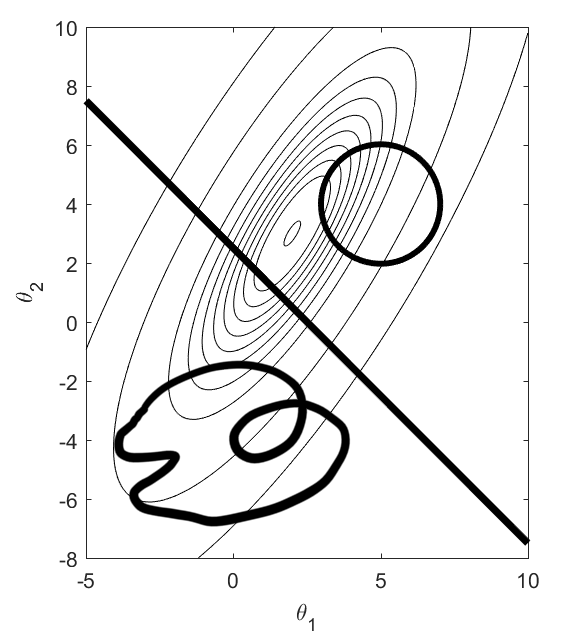}
\caption{
Level curves for a bivariate fiducial together with 
three possible curves for restriction on the parameter space.
}
\label{fig1}
\end{figure}
%

For the first case Fisher derives a fiducial by first proving 
that the projection on the
line is a sufficient statistic.
This follows from the factorization theorem.
The resulting fiducial is the normal distribution on the line centered at
the projection point with unit variance.

For the second case Fisher derives a fiducial by conditioning on the ancillary given by the
distance $a$ from the center of the circle to the observed point.
The resulting fiducial is the von Mises distribution centered at the projection point
with a concentration parameter $\kappa = a R$
\citep{MardiaJupp00statman}.

For the third and most general case \citet[p.142]{FISHER} writes:
\begin{quote}
{\em
In such cases rational
inference is effectively completed by the calculation
of the Mathematical Likelihood for each plausible
position of the unknown point.
}
\end{quote}
In the first two cases, and also in the general it seems, 
the fiducial as identified by Fisher is simply given by a
normalization of the likelihood to the given curve.
In the first two cases the fiducial is also given by the conditional distribution of the
unrestricted fiducial to the given curve.
The statistic conditioned on must have the set of parallel lines respectively 
concentric circles as level sets,
and this defines then a unique conditional fiducial model.
Furthermore, the fiducial coincides with the Bayesian posterior from the uniform prior,
and is also a confidence distribution in these two cases.

Different conditional laws on these curves can be obtained without this requirement.
This was demonstrated originally by \citet{KOLMOGOROV}
in his marginalization paradox
on the sphere, but the same holds in the plane.
Consider in particular the case where the coordinates are chosen such that
the line is given by $\theta_1=\theta_2=\mu$.
The result of Fisher is then obtained by conditioning on $\theta_1 - \theta_2 = 0$.
A different result is obtained by conditioning on $\theta_1/\theta_2 = 1$.
The latter is the choice that follows from a model
where $\mu$ is a scale parameter.

It can be concluded more generally that the fiducial is uniquely given
if the restricting condition is specified by the
value of a particular parameter as just exemplified.
A source for non-uniqueness is then given by the choice of a parameter to condition on.
The parameter $\theta_1 - \theta_2$ and the parameter $\theta_1/\theta_2$ can both
be used to ensure that the fiducial equation has a solution,
but the resulting fiducials on the line $\theta_1 = \theta_2$ are different. 

Fisher also considered the case of a semi-infinite line in the plane.
In this case he argues differently and concludes 
that the fiducial is given
by the normal distribution on the line as before,
but with a point mass at the terminal point to ensure normalization
\citep[p.140,eq.116]{FISHER}.
This is a confidence distribution corresponding to a simple
restriction of the usual unrestricted confidence intervals.

An alternative fiducial, and this coincides with the definition given before 
equation~(\ref{eqCondFid}), can be obtained by conditioning on that
the fiducial equation for the line problem should have a solution on the semi-infinite line.
This alternative fiducial is not a confidence distribution,
but it is in this case a Bayesian posterior with corresponding optimality properties.
This case of a semi-infinite line in the plane was the motivation for the
simpler case treated in subsection~\ref{sTwoFiducials}.

\subsection{Repeated sampling}
\label{ssRepeated}

\vekk{

Consider the case where 
the observation $x$ is given by
a fiducial model
\be{FishLoc}
x = \theta + u
\ee
where $x$, $\theta$ and $u$ belong to a Hilbert space.
The fiducial distribution is given by 
\be{FishLoc1}
\theta = x - u
\ee
and the Monte Carlo law for $u$.

Consider next the same model, 
but with the restriction that $\theta$
lies on the line $\theta_1=\theta_2=\ldots$.
A well defined fiducial is obtained
from the fiducial defined by equation~(\ref{eqFishLoc1})
conditionally given 
$(\theta_2-\theta_1, \theta_3-\theta_1, \ldots)=y$ evaluated at $y=0$.
An alternative fiducial is obtained by conditioning on the alternative parameter
$(\theta_2/\theta_1, \theta_3/\theta_1, \ldots)=y$ evaluated at $y=(1,1,\ldots)$.
The latter is only possibly in the finite dimensional case.

In any case, it follows that to obtain uniqueness the parameter to be conditioned on
must be specified as a part of the fiducial model.
The argument generalizes directly, 
and gives the concept of 
{\em a conditional simple fiducial model}.
}

The case of {\it iid} sampling from a distribution with a scalar parameter
can be identified with a conditional simple fiducial model as explained next.
Consider a family of strictly monotonic fiducial models 
\be{Hannig}
x_i = \chi (u_i, \theta_i)
\ee
where each component model is as in Theorem~\ref{theo1}.
The corresponding joint fiducial model for 
$x = (x_1, \ldots, x_n)$ has a corresponding conditional fiducial model
defined by the condition
$C(\theta) = (\theta_2/\theta_1, \theta_3/\theta_1,\ldots) = (1,\ldots,1)$.
This function $C$ and the value $c=(1,\ldots,1)$ is here assumed to be a part of the specification of the
conditional fiducial model.
The resulting fiducial is unique from the fiducial model~(\ref{eqHannig}) and
the condition $C=c$.

For this case it is possible to calculate the fiducial density
by introducing new coordinates $\alpha_1=\theta_1, \alpha_2=\theta_2/\theta_1, \ldots$.
The resulting fiducial density for the conditional fiducial model is
\be{gt}
c(\alpha \st x) = K \alpha^{n-1} \prod_i \abs{\partial_{\alpha} F(x_i \st \alpha)}
\ee
where $K$ is a normalization constant, 
and $\alpha^{n-1}$ is the Jacobian 
from the change-of-variables.

If the condition is replaced by
$C(\theta)=(\theta_2-\theta_1,\theta_3-\theta_1,\ldots)=0$,
then the above $K \alpha^{n-1}$ must be replaced by $K$.
It can be noted that the resulting fiducial distributions
that follow by these two arguments 
coincide with classical results for scale and location models.
This is no coincidence, and will be discussed further
in subsection~\ref{ssFraser}.

An alternative concrete example is given by equation~(\ref{eqFidBartlett1}) 
for independent samples from the gamma distribution with known scale $\beta$
and unknown shape parameter $\theta_i = \alpha$.
The result is then two alternative fiducial distributions for the shape
parameter in addition to the one resulting from the Bartlett statistic.
We do not know if there exists a third condition $C$ which would reproduce
the fiducial from the Bartlett statistic. 

Let $f=F'$ be the corresponding density of the statistical model
that follows from equation~(\ref{eqHannig}).
\citet[p.506]{Hannig09fiducial} recommends to use the following
fiducial density:
\be{Hannig1}
c (\alpha \st x) = 
K \sum_i f(x_1 \st \alpha) \cdots f(x_{i-1} \st \alpha) 
\abs{\partial_\alpha F(x_i \st \alpha)}
f(x_{i+1} \st \alpha) \cdots f(x_n \st \alpha) 
\ee
It is obtained
by formally applying Bayes rule with a data dependent prior
\be{Hannig1p}
\pi (\alpha,x) = \sum_i \frac{\abs{\partial_\alpha F(x_i \st \alpha)}}{f(x_i \st \alpha)} 
\ee
This formula was obtained from the fiducial model~(\ref{eqHannig})
by additional assumptions and suitable changes of variables to simplify
the conditioning.
The conditioning is defined differently by \citet{Hannig13fiducial}:
Instead of conditioning on $\Theta$ he conditions on
$U$ given that equation~(\ref{eqHannig}) should have a solution with
$\alpha=\theta_i$ for all $i$.
A related approach is
discussed further in subsection~\ref{ssFraser} below.
In all cases one has to add assumptions to equation~(\ref{eqHannig}) 
in order to get a unique fiducial.

It is noteworthy that in the case $n=1$ the formula~(\ref{eqHannig1})
reduces to the one given by \citet{Fisher30}, 
and it is uniquely given by the statistical model.
It can also be noted that the resulting fiducial from equation~(\ref{eqHannig1}) 
in location and scale models coincide with the results found above. 

The previous gives 
four distinct alternative fiducial distributions for the
shape parameter of the gamma distribution.
The fiducial described initially and obtained from the Bartlett statistic
in subsection~\ref{ssGamma} has the
advantage of being an exact confidence distribution and the
presented algorithm gives {\it iid} samples directly.
Samples from the three alternatives here can be produced by
Markov Chain Monte Carlo methods.

\subsection{Group models and conditioning}
\label{ssFraser}

Consider a fiducial model 
\begin{equation}
\label{eqFidEqAgain}
x = \chi (u, \theta)
\end{equation}
as in equation~(\ref{eqFidEq}),
but now in a situation where
a solution $\hat{\theta} (u,x)$ need not exist for all $(u,x)$.
One way to solve this is to introduce an enlarged fiducial model
\begin{equation}
\label{eqFidEqLarge}
x = \chi^* (u, \theta^*)
\end{equation}
with 
$\theta^* \in \Omega_{\Theta^*} \supset \Omega_\Theta$, 
$\chi^* (u,\theta) = \chi (u,\theta)$,
a guaranteed solution $\hat{\theta}^* (u,x)$,
and a condition $C (\theta^*) = c$ such that
\begin{equation}
\label{eqFidEqCondL}
\Omega_\Theta = (C=c) = \{\theta^* \st C(\theta^*)=c\}
\end{equation}
This gives a fiducial distribution,
but it depends on the specific construction and in particular on
the condition $C=c$ as exemplified in the previous two subsections.

An example of this approach is given by the fiducial model
\begin{equation}
\label{eqFidLinear}
x = u + \theta
\end{equation}
where $x$ and $u$ belongs to a Hilbert space $\Omega_X$,
and $\theta$ belong to a subspace $\Omega_\Theta \subset \Omega_X$.
In this case a solution $\theta$ 
does not exist for all $(u,x)$,
but a solution $\hat{\theta}^* (u,x) = x - u$
exists in $\Omega_{\Theta^*} = \Omega_X$.
The corresponding conditional enlarged natural fiducial model
gives the fiducial
\begin{equation}
\label{eqFidCondLinear}
(\Theta^* \st C(\Theta^*)=0) \sim
(x - U \st C(x -U)=0) \sim
(x - U \st C(U)=C(x))
\end{equation}
where $C$ is the projection on the orthogonal complement
of $\Omega_\Theta$. 

In the previous example it was demonstrated that a condition on $\Theta^*$
was equivalent with a condition on $U$.
This will always be the case if the extended model~(\ref{eqFidEqLarge})
defines a one-one correspondence between $u$ and $\theta^*$ for each $x$.
The conditioning in the example is also equivalent to
the conditional statistical model obtained by the condition
\begin{equation}
\label{eqCondStat}
C(X) = c(x)
\end{equation}
since $C(X) = C(U+\theta)=C(U)$.
This is then an example of conditional inference given an ancillary statistic $C$.
Conditional inference can be debatable in general,
but in the present case the result is still optimal inference in a 
frequentist sense \citep{TaraldsenLindqvist13fidopt}.

The previous example for a subspace of a Hilbert space
generalizes verbatim to the case of a group model where 
the parameter space is a subgroup acting on a larger group
and $C$ is a maximal invariant group homomorphism
\citep{TaraldsenLindqvist13fidopt}.

\section{Discussion}
\label{sDiscussion}


Historically, up to the present date, 
the most important development motivated by
the seminal paper by \citet{Fisher30}
is given by the theory of confidence intervals 
and hypothesis testing as developed by \citet{Neyman37confinterval}.
\citet{Birnbaum61} introduced the concept of a confidence curve
to represent confidence intervals or regions simultaneously for all confidence levels
as suggested by others \citep{Tukey49,Cox58}. 
Based on these ideas \citet{Schweder07cc} introduces
the concept of a  confidence  net $N$
which is  a  stochastic  function from  parameter
space  to  the  unit  interval.
The defining assumption
$N (\theta) \sim \Uniform (0,1)$ ensures
that $(N \le \alpha) = \{\theta \st N (\theta) \le \alpha\}$ 
is an $\alpha$-level confidence region.

An example of a confidence net is
given by $\theta \mapsto N(\theta) = F(X; \theta)$ from 
equation~(\ref{eqF1}) and the assumptions made by Fisher.
A reformulation along the lines of \citet{Blaker00cc}
shows that this can be seen as a direct generalization of
equation~(\ref{eqF1}) in that both the parameter and the data can be completely arbitrary.
Both \citet{Blaker00cc} and \citet{Schweder07cc} give
explicit and quite general methods that can be used 
for the construction of confidence nets. 

The previous gives the initial ingredients for a general theory of confidence
distributions for a general parameter space.
A confidence net can be obtained from a confidence distribution together with a
specific confidence region method.
Examples can be given to show that there exist many different confidence nets arising from
a given confidence distribution.

The concept of a fiducial distribution, 
as defined in this paper in Section~\ref{sUnique},
can be used to produce examples of confidence distributions in multiparameter cases.
We prefer, however, to use the term fiducial also for cases
where it is not interpretable as a confidence distribution.
This is in harmony with the final words written on the subject matter
by \citet[p.54-55]{FISHER}: 

\begin{quote}
By  contrast, the fiducial argument uses
the observations only to change  the  logical status of
the parameter from one in which nothing is known
of  it, and no probability statement about  it  can
be made,  to  the status of  a  random variable having
a  well-defined distribution.
\end{quote}

One of the best textbook source on the fiducial controversy known to the authors is the classic text by
\citet{StuartOrdArnold99kendal2a}.
They discuss in particular various
non-uniqueness controversies arising from the fiducial argument when generalized
to the multiparameter case.
One line of arguments focus on pivotal quantities,
and the above discussion of confidence nets can be seen as a
continuation of this line of thoughts.
Uniqueness of a confidence distribution is not generally possible,
but in good situations it can be unique given suitable additional
optimality criteria \citep{SchwederHjort16book}.
An alternative line of arguments focus on the fiducial
as a substitute for a Bayesian posterior in cases without prior information.
Again, general uniqueness seem out of reach, 
and work remains to be done on foundational issues.

Equation~(\ref{eqFidEq}) represents a common ingredient in most recent 
\citep{HannigIyerLaiLee16review}, 
and also many older \citep{FRASER,Dempster68bayesgen,Fraser79inferenceAndLinear,DawidStone82} 
developments of the original fiducial argument of \cite{Fisher30}.
There are essentially two basic complications that may arise: 
\begin{enumerate}
\item The solution with respect to the parameter is non-unique.
\item There is no solution.
\end{enumerate}
Complication 1 can be seen as the source for the development of the Dempster-Shafer theory \citep{Shafer08belief}.
An alternative solution \citep{Hannig13fiducial} is given by the additional introduction of a 
randomization rule that chooses among the
solutions allowed by equation~(\ref{eqFidEq}).

Complication 2 can intuitively be solved by conditioning on 
the existence of a solution.
This is, unfortunately, in many cases of interest technically challenging since the conditioning event may have zero probability
and non-uniqueness may arise as in the original Borel paradox described by \citet{KOLMOGOROV}.
The most appealing approach known to the authors for solving this is represented by
work by \citet{Hannig13fiducial}.
He considers the model as a limit of discretized models.
It has been demonstrated convincingly in many papers  that
this can lead to very good inference procedures in problems otherwise out of reach \citep{HannigIyerLaiLee16review}.

The conditional fiducial models introduced here can in certain cases, as a by-product, 
present an alternative solution to complication 2 as shown.
Existence and uniqueness follows then without the technical measure theoretic complications inherent in
the more general and applicable approach presented by \citet{HannigIyerLaiLee16review}. 
The definition of a conditional fiducial model as given before equation~(\ref{eqCondFid}) is,
however, not intended as a general solution to complication 2.
The main motivation behind the definition is as described in relation to the original examples given by \citet{FISHER}. 

\vekk{    
A different kind of uniqueness is obtained for conditional fiducial models: A simple fiducial model together with a condition gives a conditional fiducial model where the resulting fiducial is unique. It is NOT uniquely given by the statistical model, but it is uniquely given by the conditional fiducial model. It is demonstrated that different conditions, even though the resulting restriction in the parameter space is the same, lead to different fiducial distributions. As a byproduct, it is shown that certain 'overdetermined' fiducial models can be reformulated as conditional fiducial models. This gives an alternative in some cases to the more general approach presented by Hannig and others which obtains a solution in these cases by consideration of limits of discretized models. Existence and uniqueness of these limits is technically more difficult to handle, but it gives an intuitively appealing approach. The concept of a conditional fiducial model is not intended as a replacement of this approach, but it gives an alternative perspective. The choice of limiting approach gives rise to different fiducial distributions. This is similar to the more elementary observation that different conditions in a conditional fiducial model gives different fiducial distributions. The concept of a conditional fiducial model is, however, not introduced mainly to solve the problem of 'overdetermined' fiducial models. It is intended, as suggested initially by Fisher, as a model for the case where we have the additional information given by the condition.
}

According to \citet[p.147]{Pedersen78fiducialDead}
the fiducial argument 
has had a very limited success and was then in 1978 essentially dead.  
Now, in 2017, there are several recent publications
indicating that the fiducial argument is still 
very much alive 
and can be used successfully in both theoretical 
and practical directions
\citep{Efron98,
SchwederHjort02Confidence,
LidongHannigIyer08OneWay,
WangIyer09gumFiducial,
Frenkel09fiducialfiniteresolution,
TaraldsenLindqvist13fidopt,
XieSingh13confidence,
Hannig13fiducial,
MartinLiu13fiducialbirnbaum,
VeroneseMelilli14fiducial,
YangLiuLiuXieHoaglin14confdmeta,
NadarajahBityukovKrasnikov15confdistreview,
TaraldsenLindqvist15fidpost,
HannigIyerLaiLee16review}.
The main contributions in this paper can be summarized by the following.
\begin{itemize}
\item  
Uniqueness of the fiducial distribution in a restricted class of models is
proved in Theorem~\ref{theo1}. 
\item 
The concept of a conditional fiducial model is introduced.
\item 
Examples relating the above two points have been presented 
together with examples that show its historical origin by 
reference to the final work on this by \citet{FISHER}.   
\end{itemize}

\vekk{
\section{Acknowledgments}

The authors acknowledge constructive remarks from two referees.
It was commented that a related uniqueness result was obtained independently 
as a part of the dissertation of Abhishek Pal Majumder defended and recorded in April 2015.
The earliest similar uniqueness result known to the authors is found in the
work of \citet{DawidStone82}. 
}


\section*{References}


\bibliography{bibJA,gtaralds}
\bibliographystyle{plainnat} 

\end{document}